\documentclass{amsart}
\usepackage{graphicx}
\usepackage{amssymb}
\usepackage{amsmath}
\usepackage{enumerate}
\usepackage{comment}

\theoremstyle{plain}
\newtheorem{theorem}{Theorem}[section]
\newtheorem{theorem*}{Theorem}[section]
\newtheorem{lemma}[theorem]{Lemma}
\newtheorem{claim}[theorem]{Claim}

\newtheorem{corollary}[theorem]{Corollary}
\newtheorem{corollary*}[theorem]{Corollary}
\newtheorem{proposition}[theorem]{Proposition}
\theoremstyle{definition}
\newtheorem{definition}[theorem]{Definition}
\newtheorem{Example}[theorem]{Example}
\theoremstyle{remark}
\newtheorem*{remark}{Remark}

\def\d{\displaystyle}

\def\t{\theta}
\newcommand\gl{\mathfrak{gl}}
\def\bigz{\mathbb{Z}}
\def\lo{\underline{o}}
\def\M{\mathcal{M}}
\def\C{\mathcal{C}}
\def\N{\mathcal{N}}
\def\tK{\Bbbk((\t))}
\def\O{\text{Ord}}

\begin{document}

\normalsize

\title{Calculation of local formal Mellin transforms}
\author{Adam Graham-Squire}
\date{}

 \begin{abstract} Much recent work has been done on the local Fourier transforms for connections on the punctured formal disk.  Specifically, the local Fourier transforms have been introduced, shown to induce certain equivalences of categories, and explicit formulas have been found to calculate them.  In this paper, we prove analogous results in a similar situation, the local Mellin transforms for connections on the punctured formal disk.  Specifically, we introduce the local Mellin transforms and show that they induce equivalences between certain categories of vector spaces with connection and vector spaces with invertible difference operators, as well as find formulas for explicit calculation in the same spirit as the calculations for the local Fourier transforms.
 \end{abstract}

\maketitle

 Keywords: local integral transforms; Mellin transform; difference operator; differential operator
 
Mathematics Subject Classification 2000: 14B20

\section{Introduction}

Recently, much research has been done on local Fourier transforms for connections on the punctured formal disk.  Namely, H.~Bloch and H.~Esnault in \cite{bloch} and R. Garcia Lopez in \cite{garcialopez}  introduce and analyse the local Fourier transforms.  Explicit formulas for calculation of the local Fourier transforms were proved independently by J.~Fang in \cite{fang} and C.~Sabbah in \cite{sabbah} using different methods.  In \cite{dima}, D.~Arinkin gives a different framework for the local Fourier transforms and also gives explicit calculation of the Katz-Radon transform.  In \cite{gs}, we use Arinkin's techniques from \cite{dima} to reproduce the calculations of \cite{fang} and \cite{sabbah}. The global \emph{Mellin} transform for connections on a punctured formal disk is given by Laumon in \cite{laumonmellin} as well as Loeser and Sabbah in \cite{loeser}, but since that time little work has been done on the Mellin transform in this area.  In \cite[Section 2.5]{dima}, Arinkin remarks that it would be interesting to apply his methods to other integral transforms such as the Mellin transform. This paper is the answer to that query.  We introduce the \emph{local} Mellin transforms on the punctured formal disk and prove results for them which are analogous to those of the local Fourier transforms.  One main difference between the analysis of the local Fourier and local Mellin transforms is this- whereas the local Fourier transforms deal only with differential operators, the local Mellin transforms input a differential operator and output a difference operator. 

The work done in this paper is as follows: after some preliminary definitions, we introduce the local Mellin transforms $\M^{(0,\infty)}$, $\M^{(x,\infty)}$, and $\M^{(\infty,\infty)}$ for connections on the punctured formal disk.  Our construction of the local Mellin transforms is analogous to the work of \cite{bloch} and \cite{dima} for the local Fourier transforms.  In particular, we mimic the framework given in \cite{dima} to define the local Mellin transforms, as Arinkin's construction lends itself most easily to calculation. We also show that the local Mellin transforms induce equivalences between certain categories of vector spaces with connection and categories of vector spaces with difference operators. Such equivalences could, in principle, reduce questions about difference operators to questions about (relatively more-studied) connections, although we do not do such an analysis in this work.  We end by using the techniques of \cite{gs} to give explicit formulas for calculation of the local Mellin transforms in the same spirit of the results of \cite{fang}, \cite{gs} and \cite{sabbah}.  An example of our main result is the calculation of $\mathcal{M}^{(0,\infty)}$.  We give it here, and it is also found near the end of the paper as Theorem \ref{thmMzi}:

Let $\Bbbk$ be an algebraically closed field of characteristic zero.  Definitions for $R$, $S$, $E$, and $D$ can be found in the body of the paper.

\begin{theorem*} Let $s$ and $r$ be positive integers, $a\in\Bbbk-\{0\}$, and $f\in R^{\circ}_r(z)$ with $f=az^{-s/r}+\underline{o}(z^{-s/r})$.  Then
\[\mathcal{M}^{(0,\infty)}(E_f)\simeq D_g,\]
where $g\in {S^{\circ}_{s}(\t)}$ is determined by the following system of equations:
\begin{equation*}\label{Mzisyseq1}f=-\t^{-1}
\end{equation*}
\begin{equation*}\label{Mzisyseq2} g=z-(-a)^{r/s}\left(\frac{r+s}{2s}\right)\t^{1+(r/s)}
\end{equation*}
\end{theorem*}

A necessary tool for the calculation is the formal reduction of differential operators as well as the formal reduction of linear difference operators.  There are considerable parallels between difference operators and connections, and we refer the reader to \cite{vander2} for more details.

\subsection{Acknowledgements} The author would like to thank Dima Arinkin for his continuing assistance, support, and encouragement of this work.

\section{Connections and Difference Operators}\label{defs sec}

\subsection{Connections on the formal disk} 

\begin{definition} \label{connection} Let $V$ be a finite-dimensional vector space over $K=\Bbbk((z))$. A \emph{connection} on $V$ is a $\Bbbk$-linear operator
$\nabla:V\to V$ satisfying the Leibniz identity:
\[\nabla(fv)=f\nabla(v)+\frac{df}{dz}v\]
for all $f\in K$ and $v\in V$.  
 A choice of basis in $V$ gives an isomorphism $V\simeq K^n$; we can then write $\nabla$ as $\frac{d}{dz}+A$, where
$A=A(z)\in\gl_n(K)$ is the \emph{matrix} of $\nabla$ with respect to this basis.

\end{definition}

\begin{definition}\label{cat vs with connection}
We write $\mathcal{C}$ for the \emph{category of vector spaces with connections over} $K$. Its objects are pairs $(V, \nabla)$, where $V$ is a finite-dimensional $K$-vector space and $\nabla:V\to V$ is a connection. Morphisms between $(V_1,\nabla_1)$ and $(V_2,\nabla_2)$ are $K$-linear maps $\phi:V_1\to V_2$ that are \emph{horizontal} in the sense that $\phi\nabla_1=\nabla_2\phi$.
\end{definition}

\subsection{Properties of connections}

 We summarize below some well-known properties of connections on the formal disk. The results go back to Turrittin \cite{turritin} and Levelt \cite{levelt}; more recent references include \cite{varadar}, \cite[Sections 5.9 and 5.10]{beilinson}, \cite{malgrange}, and \cite{vander2}.

Let $q$ be a positive integer and define $K_q:=\Bbbk((z^{1/q}))$. Note that $K_q$ is the unique extension of $K$ of degree $q$.
For every $f\in K_q$, we define an object $E_f\in\mathcal{C}$ by
\[E_f=E_{f,q}=\left(K_q,\frac{d}{dz}+z^{-1}f\right).\]

In terms of the isomorphism class of an object $E_f$, the reduction procedures of \cite{turritin} and \cite{levelt} imply that we need only consider $f$ in the quotient
\begin{equation*}\label{conn equation}
\Bbbk((z^{1/q}))\Big/\left(z^{1/q}\Bbbk[[z^{1/q}]]+\frac{\mathbb{Z}}{q}\right)
\end{equation*}
where $\Bbbk[[z]]$ denotes formal \emph{power} series.
 
Let $R_q$ 
be the set of orbits for the action of the Galois group $\mathrm{Gal}(K_q/K)$ on the quotient\label{conn quotient}. Explicitly, the
Galois group is identified with the group of degree $q$ roots of unity $\eta\in \Bbbk$; the action on $f\in R_q$ is by
$f(z^{1/q})\mapsto f(\eta z^{1/q})$. Finally, denote by $R^\circ_q\subset R_q$ the set of $f\in R_q$ that cannot be represented
by elements of $K_r$ for any $0<r<q$.  Thus  $R^\circ_q$ is the locus of $R_q$ where Gal($K_q/K$) acts freely.

\begin{proposition}\label{prop} \mbox{}
\begin{enumerate}
\item \label{prop1} The isomorphism class of $E_f$ depends only on the orbit of the image of $f$ in $R_q$.
\item \label{prop2} $E_f$ is irreducible if and only if the image of $f$ in $R_q$ belongs to $R^\circ_q$. As $q$ and $f$ vary,
we obtain a complete list of isomorphism classes of irreducible objects of $\mathcal{C}$.
\item \label{prop3} Every $E\in\mathcal{C}$ can be written as
\[E\simeq\bigoplus_i(E_{f_i,q_i}\otimes J_{m_i}),\]
where the $E_{f,q}$ are irreducible, $J_m=(K^m,\frac{d}{dz}+z^{-1}N_m)$, and $N_m$ is the nilpotent Jordan block of size $m$.
\end{enumerate}
\end{proposition}

Proofs of the proposition are either prevalent in the literature (cf. \cite{beilinson}, \cite{malgrange}, \cite{vander2}) or straightforward and thus are omitted. 

\begin{remark}
  We refer to the objects $(E_f\otimes J_m)\in \C$ as \emph{indecomposable} objects in $\C$.
\end{remark}

\subsection{Difference operators on the formal disk}\label{diffop sec}

Vector spaces with difference operator and vector spaces with connection are defined in a similar fashion.

\begin{definition} \label{diffop} Let $V$ be a finite-dimensional vector space over $K=\Bbbk((\t))$. A \emph{difference operator} on $V$ is a $\Bbbk$-linear operator
$\Phi:V\to V$ satisfying
\[\Phi(fv)=\varphi(f)\Phi(v)\]
for all $f\in K$ and $v\in V$, with $\varphi:K^n\to K^n$ as the $\Bbbk$-automorphism defined below. A choice of basis in $V$ gives an isomorphism $V\simeq K^n$; we can then write $\Phi$ as $A\varphi$, where
$A=A(\t)\in\gl_n(K)$ is the \emph{matrix} of $\Phi$ with respect to this basis, and for $v(\t)\in K^n$ we have 
$$\varphi(v(\t))=v\left(\frac{\t}{1+\t}\right)=v\left(\sum^{\infty}_{i=1}(-1)^{i+1}\t^i\right).$$

We follow the convention of \cite[Section 1]{praagman} to define $\varphi$ over the extension $K_q=\Bbbk((\t^{1/q}))$. Thus for all $q\in \mathbb{Z}^+$, $\varphi$ extends to a $\Bbbk$-automorphism of $K_q^n$ defined by 
$$\varphi\Big(v(\t^{1/q})\Big)=v\left(\t^{1/q}\left[\sum^{\infty}_{i=0}\binom{-1/q}{i}\t^i\right]\right).$$
\end{definition}

\begin{definition}\label{cat vs with diffop}
We write $\mathcal{N}$ for the \emph{category of vector spaces with invertible difference operator over} $K$. Its objects are pairs $(V, \Phi)$, where $V$ is a finite-dimensional $K$-vector space and $\Phi:V\to V$ is an invertible difference operator. Morphisms between $(V_1,\Phi_1)$ and $(V_2,\Phi_2)$ are $K$-linear maps $\phi:V_1\to V_2$ such that $\phi\Phi_1=\Phi_2\phi$.
\end{definition}

\subsection{Properties of difference operators}\label{diffop properties}

In \cite{chen} and \cite{praagman}, a canonical form for difference operators is constructed. We give an equivalent construction in the theorem below, which is a restatement of \cite[Theorem 8 and Corollary 9]{praagman} with different notation so as to better fit our situation. 

\begin{theorem}[\cite{praagman}, Theorem 8 and Corollary 9]\label{diffop canon form}
Let $\Phi:V\to V$ be an invertible difference operator.  Then there exists a finite (Galois) extension $K_q$ of $K$ and a basis of $K_q\otimes_K V$ such that $\Phi$ is expressed as a diagonal block matrix. Each block is of the form
\[
F_g=\begin{bmatrix}
	g & & \\
	\t^{\lambda +1} & \ddots &\\
	 & \ddots & \ddots
	\end{bmatrix}
	\]
with $g\in K_q$, $\lambda\in \frac{\bigz}{q}$, $g=a_0\t^{\lambda}+\dots+a_q\t^{\lambda+1}$, $a_0\neq 0$, and $a_q$ defined up to a shift by $\frac{a_0\bigz}{q}\t^{\lambda+1}$.  The matrix is unique modulo the order of the blocks.
\end{theorem}

\begin{remark}
  The $F_g$ are the indecomposable components for the matrix of $\Phi$.
\end{remark}

 Theorem \ref{diffop canon form} allows us to describe the category $\N$ in a fashion similar to our description of the category $\C$.  For every $g\in K_q$, we define an object $D_g\in \N$ by
\[D_g=D_{g,q}:= \Big(K_q, g\varphi\Big). \]
The canonical form given in Theorem \ref{diffop canon form} implies that we need only consider $g$ in the following quotient of the multiplicative group $\Bbbk((\t^{1/q}))^*$:
\begin{equation}\label{kq subset}
K_q^*\Big/\left(1+\frac{\bigz}{q} \t+\t^{1+(1/q)}\Bbbk[[\t^{1/q}]]\right).
\end{equation}
 Let $S_q$ be the set of orbits for the action of the Galois group $\mathrm{Gal}(K_q/K)$ on the quotient given in \eqref{kq subset}. Denote by $S^\circ_q\subset S_q$ the set of $g\in S_q$ that cannot be represented by elements of $K_r$ for any $0<r<q$.  As before, $S^\circ_q$ can be thought of as the locus where $\mathrm{Gal}(K_q/K)$ acts freely.

\begin{proposition}\label{diffopprop} \mbox{}
\begin{enumerate}
\item \label{diffopprop1} The isomorphism class of $D_g$ depends only on the orbit of the image of $g$ in $S_q$.
\item \label{diffopprop2} $D_g$ is irreducible if and only if the image of $g$ in $S_q$ belongs to $S^\circ_q$. As $q$ and $g$ vary,
we obtain a complete list of isomorphism classes of irreducible objects of $\mathcal{N}$.
\item \label{difopprop3} Every $D\in\mathcal{N}$ can be written as
\[D\simeq\bigoplus_i(D_{g_i,q_i}\otimes T_{m_i}),\]
where the $D_{g,q}$ are irreducible, $T_m=(K^m,U_m\varphi)$, and $U_m=I_m+\t N_m$.
\end{enumerate}
\end{proposition}

\subsection{Notation}\label{notation sec}

 At times it is useful to keep track of the choice of local coordinate for $\mathcal{C}$ and $\N$, and we denote this with a subscript.  To stress the coordinate, we write $\mathcal{C}_0$ to indicate the coordinate $z$ at the point zero, $\C_x$ to indicate the coordinate $z-x:=z_x$ at a point $x\neq 0$, and $\mathcal{C}_{\infty}$ to indicate the coordinate $\zeta=\frac{1}{z}$ at the point at infinity.  Note that $\mathcal{C}_0$, $\C_x$ and $\mathcal{C}_{\infty}$ are all isomorphic to $\mathcal{C}$, but not canonically.  Similarly we can write $\N_{\infty}$ to indicate that we are considering $\N$ with local coordinate at infinity.  Since we only work with the point at infinity for $\N$, though, we generally omit the subscript.
 
 We also have a superscript notation for categories, but our conventions for the categories $\C$ and $\N$ are different and a potential source of confusion.  Superscript notation for vector spaces with connection is well-established, and the superscript corresponds to $slope$ (for a formal definition of slope, see \cite{katz}).  Thus, for example, we denote by $\mathcal{C}_{\infty}^{<1}$  the full subcategory of $\mathcal{C}_{\infty}$ of connections whose irreducible components all have slopes less than one ; that is, $E_f$ such that $-1< \text{ord}(f)$.

The correspondence to slope makes sense in the context of connections because all connections have nonnegative slope (i.e.~for all $E_f$ we have ord$(f)\leq 0$).  For difference operators we have no such restriction on the order of $f$, though, and thus a correspondence to slope would be artificial.  The superscripts for difference operators therefore refer to the \emph{order of irreducible components} as opposed to the slope.  Thus, for example, the notation $\N^{>0}$ indicates the full subcategory of $\N$ of difference operators whose irreducible components $D_g$ have the property that ord$(g)>0$.

\section{Tate Vector Spaces}

\subsection{The $z$-adic topology}

\begin{definition}
We define the $z$-adic topology on the vector space $V$ as follows: a \emph{lattice} is a $\Bbbk$-subspace $L\subset V$ that is of the form $L=\bigoplus_i\Bbbk[[z]]e_i$ for some basis $e_i$ of $V$ over $K$.  Then the \emph{$z$-adic topology on $V$} is defined by letting the basis of open neighborhoods of $v\in V$ be cosets $v+L$ for all lattices $L\subset V$.
\end{definition}

\begin{remark} 
An equivalent definition for the $z$-adic topology, without reference to choice of basis, is given in \cite[Section 4.2]{dima}.  The $z$-adic topology is also equivalent to the topology induced by any norm, as described in Lemma \ref{cassels lemma}.
\end{remark}

For ease of explication, we copy the remaining definitions and results in this section from \cite[Section 5.3]{dima}. For more details on Tate vector spaces, see \cite[Section 2.7.7]{drinfeld}.

\subsection{Tate vector spaces}

\begin{definition}\label{Tate vs def}
Let $V$ be a topological vector space over $\Bbbk$, where $\Bbbk$ is equipped with the discrete topology.  $V$ is \emph{linearly compact} if it is complete, Hausdorff, and has a base of neighborhoods of zero consisting of subspaces of finite codimension. Equivalently, a linearly compact space is the topological dual of a discrete space.  $V$ is a \emph{Tate space} if it has a linearly compact open subspace.
\end{definition}  

\begin{definition}\label{Tate type def}
  A $\Bbbk[[z]]$-module $M$ is \emph{of Tate type} if there is a finitely generated submodule $M'\subset M$ such that $M/M'$ is a torsion module that is `cofinitely generated' in the sense that 

\begin{center}
dim$_\Bbbk$ Ann$_z(M/M')<\infty$, where Ann$_z(M/M')=\{m\in M/M'|zm=0\}$
\end{center}
\end{definition}

\begin{lemma}\label{Tate vectspace lemma}\mbox{}
\begin{enumerate}
\item Any finitely generated $\Bbbk[[z]]$-module $M$ is linearly compact in the $z$-adic topology.
\item Any $\Bbbk[[z]]$-module of Tate type is a Tate vector space in the $z$-adic topology.
\end{enumerate}
\end{lemma}

\begin{proposition}\label{Tateprop}
	Let $V$ be a Tate space.  Suppose an operator $Z:V\to V$ satisfies the following conditions:
	\begin{enumerate}
	\item $Z$ is continuous, open and (linearly) compact.  In other words, if $V'\subset V$ is an open linearly compact subspace, then so are $Z(V')$ and $Z^{-1}(V')$.
	\item $Z$ is contracting.  In other words, $Z^n \to 0$ in the sense that for any linearly compact subspace $V'\subset V$ and any open subspace $U\subset V$, we have $Z^n(V')\subset U$ for $n\gg 0$.
	\end{enumerate}
Then there exists a unique structure of a Tate type $\Bbbk[[z]]$-module on $V$ such that $z\in \Bbbk[[z]]$ acts as $Z$ and the topology on $V$ coincides with the $z$-adic topology.\\
\end{proposition}

\section{The norm and order of an operator}\label{norm sec}

\subsection{Definition of norm}

In the discussion of norms in this subsection we primarily follow the conventions of \cite{cassels}, though our presentation is self-contained.  Similar treatments of norms can also be found in \cite{andre} and \cite{kedlaya}. Fix a real number $\epsilon$ such that $0<\epsilon<1$.  For $\displaystyle{f=\sum_{i=k} c_i\t^{i/q}\in K_q}$ with $c_k\neq 0$, we define the order of $f$ as ord$(f):=k/q$.

\begin{definition}\label{valuation def}
Let $f\in K$.  The valuation $|\bullet|$ on $K$ is defined as 
$$|f|=\epsilon^{\text{ord}(f)}$$
with $|0|=0$.
\end{definition}
This is a non-archimedean discrete valuation, and $K$ is complete with respect to the topology induced by the valuation.

\begin{definition}\label{na norm def}
Let $V$ be a vector space over $K$. A \emph{non-archimedean norm} on $V$ is a real-valued function $||\bullet ||$ on $V$ such that the following hold:
\begin{enumerate}
\item $||v||>0$ for $v\in V-\{0\}$.
\item $||v+w||\leq$max$(||v||,||w||)$ for all $v,w\in V$.
\item $||f\cdot v||=|f|\cdot ||v||$ for $f\in K$ and $v\in V$.
\end{enumerate}
\end{definition}

\begin{Example}\label{norm example}
Let $f_i\in K$.  The function
$$||(f_1,\dots,f_n)||=\text{max}|f_i|$$
is a norm on $K^n$, and $K^n$ is complete with respect to this norm.
\end{Example}

\begin{lemma}[\cite{cassels}, lemma in Section 2.8] \label{cassels lemma}
Any two norms $||\bullet||_1$, $||\bullet||_2$ on a finite-dimensional vector space $V$ over $K$ are equivalent in the following sense: there exists a real number $C>0$ such that
$$\frac{1}{C}||\bullet||_1\leq||\bullet||_2\leq C||\bullet||_1.$$
\end{lemma}
It follows from Lemma \ref{cassels lemma} that all norms on a finite-dimensional vector space over $K$ induce the same topology.

\begin{definition}\label{op norm def}
Let $A:V\to V$ be a $\Bbbk$-linear operator.  We define the \emph{norm of an operator} to be 
\[||A||=\sup_{v\in V-\{0\}}\left\{\frac{||A(v)||}{||v||}\right\}.\]
\end{definition}

Note that $||A||<\infty$ if and only if $A$ is continuous (\cite{kolmogorov}[Chapter 6, Theorem 1]).

\subsection{Invariant norms}

The norm of an operator given in Definition \ref{op norm def} depends on the choice of the non-archimedean norm $||\bullet||$.  To find an invariant for norms of operators, consider the following two norms:

\begin{definition}
The \emph{infimum} norm is defined as 
\[ ||A||_{inf}=\inf\{||A||: ||\bullet|| \text{ is a norm on }V\}\]
and the \emph{spectral radius} of $A$ is given by
 \[\d{||A||_{spec}=\lim_{n\to\infty}\sqrt[n]{||A^n||}}.\]
\end{definition}

Note that $A$ must be continuous to guarantee that the limit defining the spectral radius exists. It follows from Lemma \ref{cassels lemma} that the spectral radius does not depend on the choice of norm $||\bullet||$. 
 For operators in general the spectral radius is often the more useful invariant, but for the class of operators we consider (connections, difference operators, and their inverses) the two definitions coincide and we primarily use the infimum norm.

\subsection{Norms of similitudes}

\begin{proposition}\label{norms geq 1}
Let $||\bullet||_1$ and $||\bullet||_2$ be two norms on $V$.  Then for any invertible $\Bbbk$-linear operator $A:V\to V$, we have $||A||_1\cdot ||A^{-1}||_2 \geq 1$.
\end{proposition}

\begin{corollary}\label{inf norm corollary}
Let $A:V\to V$ be invertible and $||\bullet||$ a norm such that $||A||\cdot ||A^{-1}||=1$. Then $||A||=||A||_{inf}$.
\end{corollary}

\begin{definition}
Let $||\bullet||$ be a norm on $V$.  An operator $A:V\to V$ is a \emph{similitude} (with respect to $||\bullet||$) if $||Av||=\lambda||v||$ for all $v\in V$.  It follows that $||A||=\lambda$.
\end{definition}
\begin{claim}\label{similitude claim}
If $A:V\to V$ is an invertible similitude with $||Av||=\lambda||v||$, then $||A||_{inf}=\lambda$ and $||A^{-1}||=\frac{1}{\lambda}$.
\end{claim}

\subsection{Properties of norms}

Given the canonical form of a connection or difference operator, it is quite easy to calculate the norm. In particular we note that indecomposable connections with no horizontal sections and indecomposable invertible difference operators are similitudes. 
\begin{remark}
We introduce here notation to clear up a potentially confusing situation.  The issue is the notation $\nabla=\frac{d}{dz}+A$ for a connection.  In particular, at the local coordinate $\zeta=\frac{1}{z}$ the change of variable gives us $\nabla=-\zeta^2\frac{d}{d\zeta}+A(\zeta)$.  To emphasize the local coordinate we will use the notation $\nabla_z$ (respectively $\nabla_{\zeta}$) to indicate that we are writing $\nabla$ in terms of $z$ (respectively $\zeta$).  In particular we have the equalities $\nabla_z=-\zeta^2\nabla_\zeta$ and $z\nabla_z=-\zeta\nabla_\zeta$.  

\end{remark}
\begin{proposition}\label{normprop}
 For an indecomposable $(V, \nabla)=(E_f\otimes J_m)\in \C$ such that $\nabla$ has no horizontal sections,
\begin{enumerate}
 \item \label{normprop1} $||\nabla||_{inf}=\epsilon^{\textup{ord}(f)-1}$.

 If $\nabla$ is invertible we also have
	
 \item \label{normprop2} $||\nabla^{-1}||_{inf}=\epsilon^{-\textup{ord}(f)+1}$.
 	\vspace{.1in}
  \item \label{normprop3} For $(V,\nabla)\in \C_0$, $||(z\nabla)^{-1}||_{inf}= \epsilon^{-\textup{ord}(f)}$.
	\vspace{.1in}
  \item \label{normprop3.5} For $(V,\nabla)\in \C_{\infty}$, $||(z\nabla)^{-1}||_{inf}=||(\zeta\nabla_\zeta)^{-1}||_{inf}=\epsilon^{-\textup{ord}(f)}$.
	\vspace{.1in}
  \item \label{normprop4} For $(V,\nabla)\in \C_x$, $||(z\nabla_{z_x})^{-1}||_{inf}= \epsilon^{1-\textup{ord}(f)}$.
\end{enumerate}
\end{proposition}

\begin{proposition}\label{diffop norm prop}
For an indecomposable $(V, \Phi)=(D_g\otimes T_m)\in \N$,
\begin{enumerate}
 \item \label{diffop normprop1} $||\Phi||_{inf}=\epsilon^{\textup{ord}(g)}$.
 \vspace{.1in}
 \item \label{diffop normprop2} $||(\t\Phi)^{-1}||_{inf}=\epsilon^{-\textup{ord}(g)-1}$.

\end{enumerate}
\end{proposition}

\subsection{Order of an operator}

The $order$ of an operator is a notion closely related to the norm of an operator. It is often more convenient to work with order as opposed to norm, so we give a brief introduction to order below.

\begin{definition}\label{order def}
Let $B:V\to V$ be a $\Bbbk$-linear operator and $||\bullet ||$ a norm defined on $V$.  Then the $order$ of $B$ is 
$$\text{Ord}(B)=\log_{\epsilon}||B||_{spec},$$
with $\O(0):=\infty$. 
\end{definition}

\begin{Example}\label{order remark}
The term ``order'' is suggestive for the following reason.  Given Definition \ref{order def}, the properties of similitudes, and $\nabla$ an indecomposable connection with no horizontal sections, the following property holds: Ord($\nabla)=\ell$ if and only if for all $n\in\mathbb{Q}$ we have $\nabla(z^n1)=(*z^{n+\ell})1+$ higher order terms, where 1 is the identity element of $V$.  Similarly for an indecomposable difference operator $\Phi$, Ord($\Phi)=j$ if and only if $\Phi(\t^n1)=(*\t^{n+j})1 +$ higher order terms.  Note that $*\in \Bbbk-\bigz$ if $\ell=-1$ and $*\in\Bbbk$ otherwise.
\end{Example}

In the context of the order of an operator, we can state the results of Propositions \ref{normprop} and \ref{diffop norm prop} as follows.
\begin{corollary}[to Propositions \ref{normprop} and \ref{diffop norm prop}]\label{ordercorollaries}
For indecomposable $(V,\nabla)=(E_f\otimes J_m)$, with $z\nabla$ invertible, in either $\C_0$ or $\C_{\infty}$ we have
\begin{enumerate}
\item \label{ordercorollaries1} \textup{Ord}$(\nabla)=\textup{ord}(f)-1$, \textup{Ord}$\Big(z\nabla\Big)=\textup{ord}(f)$, and \textup{Ord}$\Big((z\nabla)^{-1}\Big)=-\textup{ord}(f)$.\\
\vspace{.1in}
For indecomposable $(V,\nabla_{z_x})=(E_f\otimes J_m)\in \C_x$, with $z\nabla_{z_x}$ invertible,
\item \label{ordercorollaries1.5} $\O\Big(z\nabla_{z_x}\Big)=\O\Big(\nabla_{z_x}\Big)=\textup{ord}(f)-1$ and \textup{Ord}$\Big((z\nabla_{z_x})^{-1}\Big)=1-\textup{ord}(f)$.\\
\vspace{.1in}
For indecomposable $(V,\Phi)=(D_g\otimes U_m)\in\N$ 
\item \label{ordercorollaries2} $\textup{Ord}(\Phi)=\textup{ord}(g)$ and $\textup{Ord}\Big((\t\Phi)^{-1}\Big)=-\textup{ord}(g)-1$.
\end{enumerate}
\end{corollary}

\section{Lemmas}\label{lemmas sec}

\subsection{Fractional powers of an operator}

The Operator-root Lemma below shows how to calculate the power of a sum of certain operators, even for fractional powers.  The idea is that once a certain root (1/$p$) of the operator is chosen, the fractional power is easily defined as an integer power of that root.

\begin{lemma}[Operator-root Lemma]\label{abratlemma}
    Let $A$ and $B$ be the following $\Bbbk$-linear operators on $K_q$: $A=\text{ multiplication by } f= az^{p/q}+ \underline{o}(z^{p/q})$, $0\neq a \in \Bbbk$, and $B=z^n\frac{d}{dz}$ with $n\neq 0$, $p\neq 0,$ and $q> 0$ all integers.  We have Ord$(A)=\frac{p}{q}$ and  Ord$(B)=n-1$, and we assume that $\frac{p}{q}< n-1$. Then for any $p^{\text{th}}$ root of $A$ we can choose a $p^{\text{th}}$ root of $(A+B)$, $(A+B)^{1/p}$, such that 
\begin{equation*}
        (A+B)^{m}=A^{m}+mA^{(m-1)}B+\frac{m(m-1)}{2}A^{m-2}[B,A] + \underline{o}(z^{(p/q)(m-1)+n-1})
    \end{equation*} holds for all $m\in \frac{\mathbb{Z}}{p}$ where $(A+B)^{m}=((A+B)^{1/p})^{pm}$.
\end{lemma}

\begin{proof}
	A full proof is found in \cite[Lemma 4.4]{gs}.
\end{proof}

\subsection{Tate Vector Space Lemmas}

We also need some lemmas describing our situation in the language of Tate vector spaces.  The proofs are straightforward and are omitted.

\begin{lemma}\label{cont iff slope lemma}
  Let $Z:V\to V$ be a $\Bbbk$-linear operator. If Ord$(Z)>0$, then $Z$ is contracting
\end{lemma}

\begin{lemma}\label{tate vs fin dim lemma}
A $K$-vector space $V$ is of Tate-type if and only if it is finite dimensional.
\end{lemma}

\section{Global Mellin transform}

  The `classical' Mellin transform can be stated as follows:  for an appropriate $f$ the Mellin transform of $f$ is given by
\[ \tilde{f}(\eta)=\int^\infty_0 z^{\eta-1}f(z)dz
\]  
and one can check that the following identities hold:
\begin{itemize}
\item $\eta\tilde{f}=-\widetilde{\left(z\frac{d}{dz}f\right)}$
\item $\Phi\tilde{f}=\widetilde{(zf)}$
\end{itemize}
where $\Phi$ is the difference operator taking $\tilde{f}(\eta)$ to $\tilde{f}(\eta+1)$.

This leads to the notion of the $global$ Mellin transform for connections on a punctured formal disk, which was introduced by Laumon in \cite{laumonmellin} and also presented in \cite{loeser}. Below is our definition for the global Mellin transform, which is equivalent to Laumon's.

\begin{definition}
	The global Mellin transform $\M:\Bbbk[z, z^{-1}] \langle\nabla\rangle \to \Bbbk[\eta] \langle\Phi, \Phi^{-1}\rangle$ is a homomorphism between algebras defined on its generators by $-z\nabla\mapsto \eta$ and $z\mapsto \Phi$.  Note that we have $[\nabla, z]=1$ for the domain and $[\Phi,\eta]=\Phi$ for the target space, and the homomorphism preserves these equalities.
\end{definition}

 As in the case of the Fourier transform, we derive our definition of the local Mellin transform from the global situation.  In particular, the local Mellin transform has different `flavors' depending on the point of singularity, so we refer to them as local Mellin transform\emph{s}.

\section{Definitions of local Mellin transforms}\label{LMT defs sec}

Below we give definitions of the local Mellin transforms.  To alleviate potential confusion, let us explain the format we will use for the definitions.  We begin by stating the definition in its entirety, but it is not \emph{a priori} clear that all statements of the definition are true.  We then claim that the transform is in fact well-defined and give a proof to clear up the questionable parts of the definition.

\begin{definition}\label{def mellin zi}
  Let $E=(V,\nabla)\in\C_0^{>0}$. Thus all indecomposable components of $\nabla$ have slope greater than zero,  so each indecomposable component $E_f\otimes J_m$ has ord$(f)<0$.  Consider on $V$ the $\Bbbk$-linear operators
\begin{equation}\label{eq mellin zi}
  \t:=-(z\nabla)^{-1}: V\to V \text{  and  }  \Phi:=z:V\to V
  \end{equation}
Then $\t$ extends to an action of $\Bbbk((\t))$ on $V$, dim$_{\tK}V<\infty$, and $\Phi$ is an invertible difference operator.  We write $V=V_{\t}$ to denote that we are considering $V$ as a $\tK$-vector space.  We define \emph{the local Mellin transform from zero to infinity of E} to be the object
$$\M^{(0,\infty)}(E):=(V_{\t},\Phi)\in \N.$$
\end{definition}

\begin{definition}\label{def mellin xi}
  Let $E=(V,\nabla)\in\C_x$ such that $\nabla$ has no horizontal sections.  Consider on $V$ the $\Bbbk$-linear operators
\begin{equation}\label{eq mellin xi}
  \t:=-(z\nabla)^{-1}: V\to V \text{  and  }  \Phi:=z:V\to V
  \end{equation}
Then $\t$ extends to an action of $\tK$ on $V$, dim$_{\tK}V<\infty$, and $\Phi$ is an invertible difference operator.  We define \emph{the local Mellin transform from $x$ to infinity of E} to be the object
$$\M^{(x,\infty)}(E):=(V_{\t},\Phi)\in \N.$$
\end{definition}

\begin{remark}
Since $E\in \C_x$, we are thinking of $K$ as $\Bbbk((z_x))$.  This emphasizes that we are localizing at a point $x\neq 0$ with local coordinate $z_x=z-x$.
\end{remark}

Note that in the following definition we are thinking of $K$ as $\Bbbk((\zeta))$, since we are localizing at the point at infinity $\zeta=\frac{1}{z}$.

\begin{definition}\label{def mellin ii}
  Let $E=(V,\nabla)\in\C_{\infty}^{>0}$, thus all irreducible components of $\nabla$ have slope greater than zero.  Consider on $V$ the $\Bbbk$-linear operators
\begin{equation*}
  \t:=-(z\nabla)^{-1}: V\to V \text{  and  }  \Phi:=z:V\to V
  \end{equation*}
Then $\t$ extends to an action of $\tK$ on $V$, dim$_{\tK}V<\infty$, and $\Phi$ is an invertible difference operator.  We define \emph{the local Mellin transform from infinity to infinity of E} to be the object
$$\M^{(\infty,\infty)}(E):=(V_{\t},\Phi)\in \N.$$
\end{definition}

\begin{claim}\label{zi def claim}
$\M^{(0,\infty)}$ is well-defined.
\end{claim}

\begin{proof}
To prove the claim we must show the following:
\begin{enumerate}[(i)]
\item \label{claimproof1} $\t$ extends to an action of $\Bbbk((\t))$ on $V$. 
\item \label{claimproof2} $V_{\t}$ is finite dimensional.
\item \label{claimproof3} $\Phi$ is an invertible difference operator on $V_{\t}$.
\end{enumerate}

 We prove \eqref{claimproof1} with Lemma \ref{ext subclaim} below. In the proof of Lemma \ref{ext subclaim} we show that $(z\nabla)^{-1}$ satisfies the conditions of Proposition \ref{Tateprop}, and it follows that $V_{\t}$ is of Tate type.  Lemma \ref{tate vs fin dim lemma} then implies that $V_{\t}$ is finite-dimensional, proving \eqref{claimproof2}.  To prove \eqref{claimproof3}, we first note that $\Phi$ is invertible by construction.  To see that $\Phi$ is a difference operator, we need to show that $\Phi(fv)=\varphi(f)\Phi(v)$ for all $f\in K$ and $v\in V$.  Since $\Phi$ is $\Bbbk$-linear and Laurent polynomials are dense in Laurent series, this reduces to showing that $\Phi(\t^i)=\varphi(\t^i)\Phi$, which can be proved by induction so long as you can show that $\Phi(\t)=\left(\frac{\t}{1+\t}\right)\Phi$.  This last equation is equivalent to $(\eta+1)\Phi=\Phi \eta$, which we now prove. Using the fact that $[\nabla,z]=1$ and the definitions given in \eqref{eq mellin zi} we compute
\[ (\eta+1)\Phi=-z\nabla z+z=-z(z\nabla+1)+z=z(-z\nabla)=\Phi\eta.\]
\end{proof}

\begin{lemma}\label{ext subclaim}
The definition for $\t$, as given in \eqref{eq mellin zi}, extends to an action of $\Bbbk((\t))$ on $V$.
\end{lemma}

\begin{proof}
 Since all indecomposable components of $\nabla$ have positive slope, $\nabla$ (and $z\nabla$) will be invertible and thus $\t$ is well-defined.  An action of $\Bbbk[\t^{-1}]=\Bbbk[-z\nabla]$ on $V$ is trivially defined.  If $(z\nabla)^{-1}:V\to V$ satisfies the conditions of Proposition \ref{Tateprop} we will also have an action of $\Bbbk[[\t]]$ on $V$.  This will give a well defined action of $\tK$ on $V$.  Thus all we need to prove is that $(z\nabla)^{-1}:V\to V$ satisfies the conditions of Proposition \ref{Tateprop}.

We must show that $\t=(z\nabla)^{-1}:V \to V$ is continuous, open, linearly compact and contracting.  Due to the canonical form for difference operators, we can assume without loss of generality that $\nabla$ is indecomposable and $z\nabla$ is of the form
  $$ z\frac{d}{dz}+ \begin{bmatrix}
  f & &\\
  1 &\ddots &\\
   &\ddots &\ddots\\
  \end{bmatrix} $$
with $f\in \Bbbk[z^{-1/r}]$ and ord$(f)=-m/r<0$.  Let $\{e_i\}$ be the canonical basis.  Since lattices are linearly compact open subspaces, to prove that $(z\nabla)^{-1}$ is open, continuous and linearly compact it suffices to show that $(z\nabla)$ and $(z\nabla)^{-1}$ map a lattice of the form $L_k=\bigoplus (z^{1/r})^k\Bbbk[[z^{1/r}]]e_i$ to a lattice of the same form.

  We see that $z\nabla(L_k)=\bigoplus (z^{1/r})^{k-m}\Bbbk[[z^{1/r}]]e_i=L_{k-m}$ and $(z\nabla)^{-1}(L_k)=\bigoplus (z^{1/r})^{k+m}\Bbbk[[z^{1/r}]]e_i=L_{k+m}$, so $(z\nabla)^{-1}$ is open, continuous and linearly compact.

To show that $(z\nabla)^{-1}$ is contracting, by Lemma \ref{cont iff slope lemma} we only need to show $\O((z\nabla)^{-1})>0$. By Corollary \ref{ordercorollaries}, \eqref{ordercorollaries1}, then, it suffices to show that we have ord$(f)<0$ for the indecomposable $(V,\nabla)=E_f\otimes J_m$.  This condition is fulfilled by assumption, since all indecomposable components have slope greater than zero.
\end{proof}

The proof that $\M^{(x,\infty)}$ is well-defined is similar to the proof above, with only one major caveat.  In the proof of (i), we use the fact that the leading term of the operator is the only important term for the theoretical calculation.  Thus one can think of $z$ as $z_x+x$, and reduce to considering $z\nabla$ as merely $x\nabla$, from which the result readily follows.  The proofs of (ii) and (iii) are identical.  The proof that $\M^{(\infty,\infty)}$ is well-defined is virtually identical to the proof of Claim \ref{zi def claim} once the change of variable from $z$ to $\zeta$ is taken into consideration.

\begin{remark}
  Note that the local Mellin transforms above give functors to apply to all connections except for certain connections with regular singularity.  More precisely, the only invertible connections for which $\M^{(0,\infty)}$, $\M^{(x,\infty)}$, and $\M^{(\infty,\infty)}$ \emph{cannot} be applied are those connections in $\C_0$ and $\C_{\infty}$ with slope zero.  We conjecture that these connections with regular singularity will map to difference operators with singularity at a point $y\neq \infty$.  This regular singular case is sufficiently small, and the techniques necessary to prove our conjecture sufficiently different from the situation described above, that we do not discuss it here.
\end{remark}

\section{Definition of local inverse Mellin transforms}\label{ILMT defs sec}

\begin{definition}\label{def inv mellin zi}
  Let $D=(V,\Phi)\in\N^{>0}$. Thus $\Phi$ is invertible and the irreducible components of $\Phi$ have order greater than zero.  Consider on $V$ the $\Bbbk$-linear operators
\begin{equation}\label{eq mellin def inv zi}
  z:=\Phi: V\to V \text{  and  }  \nabla:=-(\t\Phi)^{-1}:V\to V
  \end{equation}
Then $z$ extends to an action of $\Bbbk((z))$ on $V$, dim$_{\Bbbk((z))}V<\infty$, and $\nabla$ is a connection.  We write $V_{z}$ for $V$ to denote that we are considering $V$ as a $\Bbbk((z))$-vector space.  We define \emph{the local inverse Mellin transform from zero to infinity of D} to be the object
$$\M^{-(0,\infty)}(D):=(V_{z},\nabla)\in \C_0.$$
\end{definition}

\begin{definition}\label{def inv mellin xi}
  Let $D=(V,\Phi)\in\N^{=0}$ such that all irreducible components of $\Phi$ have order zero with the same leading coefficient $x\neq 0$, and $\Phi-x$ is invertible.
 
  Consider on $V$ the $\Bbbk$-linear operators
\begin{equation*}\label{eq def inv mellin xi}
  z:=\Phi: V\to V \text{  and  }  \nabla:=-(\t\Phi)^{-1}:V\to V.
  \end{equation*}
Then the action of $z-x=z_x$ is clearly defined, $z_x$ extends to an action of $\Bbbk((z_x))$ on $V$, dim$_{\Bbbk((z_x))} V<\infty$, and $\nabla$ is a connection.  We write $V_{z_x}$ for $V$ to denote that we are considering $V$ as a $\Bbbk((z_x))$-vector space.  We define \emph{the local inverse Mellin transform from $x$ to infinity of D} to be the object
$$\M^{-(x,\infty)}(D):=(V_{z_x},\nabla)\in \C_x.$$
\end{definition}

\begin{definition}\label{def inv mellin ii}
  Let $D=(V,\Phi)\in\N^{<0}$.  Thus $\Phi$ is invertible and the irreducible components of $\Phi$ have order less than zero.  Consider on $V$ the $\Bbbk$-linear operators
\begin{equation*}
  z:=\Phi: V\to V \text{  and  }  \nabla:=-(\t\Phi)^{-1}:V\to V
  \end{equation*}
Then $\zeta=z^{-1}$ extends to an action of $\Bbbk((\zeta))$ on $V$ and dim$_{\Bbbk((\zeta))} V<\infty$.  We write $V_{\zeta}$ for $V$ to denote that we are considering $V$ as a $\Bbbk((\zeta))$-vector space.  We define \emph{the local inverse Mellin transform from infinity to infinity of $D$} to be the object
$$\M^{-(\infty,\infty)}(D):=(V_{\zeta},\nabla)\in \C_{\infty}.$$
\end{definition}

\begin{claim}\label{inv claim zi}
$\M^{-(0,\infty)}$ is well-defined.
\end{claim}

\begin{proof}
To prove the claim we must show the following:
\begin{enumerate}[(i)]
\item \label{claimproof1 inv zi} $z$ extends to an action of $\Bbbk((z))$ on $V$. 
\item \label{claimproof2 inv zi} $V_{z}$ is finite dimensional.
\item \label{claimproof3 inv zi} $\nabla$ is a connection on $V_{z}$.
\end{enumerate}

 We prove \eqref{claimproof1 inv zi} with Lemma \ref{ext subclaim inv zi} below. In the proof of Lemma \ref{ext subclaim inv zi} we show that $V_z$ is of Tate type.  Lemma \ref{tate vs fin dim lemma} then implies that $V_z$ is finite-dimensional, proving \eqref{claimproof2 inv zi}.  To prove \eqref{claimproof3 inv zi} we must show that $[\nabla,f]=f'$ for all $f\in\Bbbk((z))$.  Since $\nabla$ is $\Bbbk$-linear and Laurent polynomials are dense in Laurent series, to show that $[\nabla,f]=f'$ we merely need to show that $[\nabla, z^n]=nz^{n-1}$ for all $n\in\bigz$.  A straightforward calculation shows that $[\nabla,z]=1$, though, and then $[\nabla, z^n]=nz^{n-1}$ follows by induction.
 
\end{proof}

\begin{lemma}\label{ext subclaim inv zi}
 The definition of $z$, as given in \eqref{eq mellin def inv zi}, extends to an action of $\Bbbk((z))$ on $V$.
\end{lemma}

\begin{proof}
  $\Phi$ is invertible, so an action of $\Bbbk[z^{-1}]$ is defined.  We prove that $\Phi$ satisfies the conditions of Proposition \ref{Tateprop} to show that an action of $\Bbbk[[z]]$ is well-defined.
  
  To apply Proposition \ref{Tateprop}, we need to show that $z=\Phi$ is continuous, open, linearly compact, and contracting. First we show that $\Phi$ is open, continuous and linearly compact. We can assume that $\Phi$ is indecomposable, so in canonical form $(V,\Phi)=D_g\otimes T_m$ for some $g\in K_r$ with ord$(g)=s/r$.

 Let $\{e_i\}$ be the canonical basis.  As in previous proofs, it suffices to show that $\Phi$ and $\Phi^{-1}$ map a lattice of the form $L_k=\bigoplus (\t^{1/r})^kAe_i$ to a lattice of the same form (note that here we are using $A=\Bbbk[[\t^{1/r}]]$).  Calculation using the canonical form shows that $\Phi(L_k)=L_{k+s}$ and $\Phi^{-1}(L_k)=L_{k-s}$, so $\Phi$ is open, continuous and linearly compact.

To show that an indecomposable $\Phi$ is contracting, by Lemma \ref{cont iff slope lemma} we need to show that $\O(\Phi)>0$. By Corollary \ref{ordercorollaries}, \ref{ordercorollaries2}, then, we simply need to show that for $(V,\Phi)=D_g\otimes T_m$ we have ord$(g)>0$.  This follows from the assumption that all irreducible components of $\Phi$ have order greater than zero.
 
\end{proof}

The proofs that $\M^{-(x,\infty)}$ and $\M^{-(\infty,\infty)}$ are well-defined are similar and are omitted.

\section{Equivalence of categories}\label{equiv cat sec}

Assuming that composition of the functors is defined, by inspection one can see that $\M^{(0,\infty)}$ and $\M^{-(0,\infty)}$ are inverse functors (and the same holds for the pairs $\M^{(x,\infty)}$ ,$\M^{-(x,\infty)}$ and $\M^{(\infty,\infty)}$, $\M^{-(\infty,\infty)}$).  

 Thus to show that the local Mellin transforms induce certain equivalences of categories, all we need is to confirm that the functors map into the appropriate subcategories.  We first prove an important property of normed vector spaces which coincides with properties of Tate vector spaces.  This will be useful in demonstrating the equivalence of categories.

\subsection{Normed vector spaces}

Our first goal is to prove the following lemma, which will greatly simplify the relationship between the norm of an operator and its local Mellin transform.  First we give some definitions related to infinite-dimensional vector spaces over $\Bbbk$.
\begin{definition}
Let $V$ be an infinite-dimensional vector space over $\Bbbk$. A \emph{norm} on $V$ is a real-valued function $||\bullet ||$ such that the following hold:
\begin{enumerate}
\item $||v||>0$ for $v\in V-\{0\}$, $||0||=0$.
\item $||v+w||\leq$max$(||v||,||w||)$ for all $v,w\in V$.
\item $||c\cdot v||=||v||$ for $c\in \Bbbk$ and $v\in V$.
\end{enumerate}
\end{definition}

Note that the above definition applies to an infinite-dimensional vector space over $\Bbbk$, as opposed to $K$.  Thus it is similar to Definition \ref{na norm def}, but not the same.

\begin{definition}
An infinite-dimensional vector space $V$ over $\Bbbk$ is \emph{locally linearly compact} if for any $r_1>r_2>0$, $r_i\in\mathbb{R}$, the ball of radius $r_2$ has finite codimension in the ball of radius $r_1$.
\end{definition}

\begin{proposition}\label{Mellin defs lemma}
Let $V$ be an infinite-dimensional vector space over $\Bbbk$, equipped with a norm $||\bullet||$ such that $V$ is complete in the induced topology.  Let $0<\epsilon<1$ and $Y:V\to V$ be an invertible $\Bbbk$-linear operator such that $||Y||=\epsilon^\alpha<1$ and $||Y^{-1}||=\epsilon^{-\alpha}$.  Define $\hat{\epsilon}:=\epsilon^\alpha$. Then
\begin{enumerate}
\item \label{Mellin defs lemma1} $V$ has a unique structure of a $K=\Bbbk((y))$-vector space such that $y$ acts as $Y$ and the norm $||\bullet||$ agrees with the valuation on $K$ where $|f|=\hat{\epsilon}^{\text{ord}(f)}$ for $f\in K$.
\item $V$ is finite-dimensional over $K$ if and only if $V$ is locally linearly compact.
\end{enumerate}
\end{proposition}

\begin{remark}
  If $V$ is a Tate vector space then the unique structure of Proposition \ref{Mellin defs lemma}, \eqref{Mellin defs lemma1} coincides with that of Proposition \ref{Tateprop}.
\end{remark}

\begin{corollary}\label{operator order relation} Let $V$ be a $\Bbbk((y))$-vector space,
$Z:V\to V$ a similitude, and $||Z||=||Z||_{inf}=\epsilon^\alpha<1$.  Then $V$ can be considered as a $\Bbbk((Z))$-vector space (in the spirit of Proposition \ref{Mellin defs lemma}

) and for any similitude $A:V\to V$ we have $||A||=||A||_Z$.  In particular, $A$ will be a similitude when $V$ is viewed as either a $\Bbbk((y))$- or $\Bbbk((Z))$-vector space.
\end{corollary}

\subsection{Lemmas}

\begin{lemma}\label{indecomp map}
  The local Mellin transforms map indecomposable objects to indecomposable objects.
\end{lemma}

\begin{proof}
  We give the proof for $\M^{(0,\infty)}$, the proofs for the others are identical.  Suppose that $\M^{(0,\infty)}(V,\nabla)=(V_{\t},\Phi)$ and $V_\t$ has a proper subspace $W$ such that $\Phi(W)\subset W$.  Since $V_{\t}$ is a $\Bbbk((\t))$-vector space we also trivially have that $\t(W)\subset W$.  By definition of $\M^{(0,\infty)}$, this means that $z(W)\subset W$ and $-(z\nabla)^{-1}(W)\subset W$.  In particular, it follows that $\nabla (W)\subset W$, so $W$ is a proper subspace of $V$ which is $\nabla$-invariant.  This implies that if the local Mellin transform of an object is decomposable, the original object is decomposable as well, and the result follows.
\end{proof}

\begin{lemma}\label{Mzi target lemma} Let $E=(V, \nabla)\in\C_0^{>0}$, $\t$, and $\Phi$ be as in Definition \ref{def mellin zi}. Then $\M^{(0,\infty)}(E)\in \N^{>0}$.
\end{lemma}

\begin{proof}
 Due to the canonical decomposition it suffices to prove the lemma when $E$ is indecomposable.  Then $\nabla$ and $z$ are similitudes, so by Corollary \ref{operator order relation}, $\t$ and $\Phi$ are also similitudes.  By  Lemma \ref{indecomp map}, $\Phi$ is indecomposable, so to prove Lemma \ref{Mzi target lemma} it suffices to show that $||\Phi||_\t<1$.

  By Corollary \ref{operator order relation} we have that $||A||_z=||A||_\t$ for any similitude $A$, and it follows that
 \[ ||\Phi||_\t=||z||_z=(\epsilon)^{1}<1. \]
\end{proof}

The remaining lemmas have proofs similar to the proof of Lemma \ref{Mzi target lemma}, and are omitted.

\begin{lemma}\label{inv Mzi target lemma} Let $D=(V,\Phi)\in \N^{>0}$ be as in Definition \ref{def inv mellin zi}. Then $\M^{-(0,\infty)}(D)\in \C_0^{>0}$.
\end{lemma}

\begin{lemma}\label{Mxi target lemma}
  Let $E=(V, \nabla)\in\C_x$ be as  in Definition \ref{def mellin xi}. Then $\M^{(x,\infty)}(E)\in \N^0$.
\end{lemma}

\begin{lemma}\label{Mii target lemma}
  Let $E=(V, \nabla)\in\C_{\infty}^{>0}$ be as in Definition \ref{def mellin ii}. Then $\M^{(\infty,\infty)}(E)\in \N^{<0}$.
\end{lemma}

\begin{lemma}\label{inv Mii target lemma} 
Let $D=(V,\Phi)\in \N^{<0}$ be as in Definition \ref{def inv mellin ii}. Then $\M^{-(\infty,\infty)}(D)\in \C_{\infty}^{>0}$.
\end{lemma}

\subsection{Proofs for equivalence of categories}

\begin{theorem}\label{Mzi equiv}
  The local Mellin transform $\M^{(0,\infty)}$ induces
   an equivalence of categories between $\C_0^{>0}$ and $\N^{>0}$.
\end{theorem}

\begin{proof}
  This follows from Lemmas \ref{Mzi target lemma} and \ref{inv Mzi target lemma}, as well as the fact (stated above) that $\M^{(0,\infty)}$ and $\M^{-(0,\infty)}$ are inverse functors.
\end{proof}

\begin{theorem}
The local Mellin transform $\M^{(x,\infty)}$ induces an equivalence of categories between the subcategory of $\C_x$ of connections with no horizontal sections and $\N^0$.
\end{theorem}

\begin{theorem}
The local Mellin transform $\M^{(\infty,\infty)}$ induces an equivalence of categories between $\C_{\infty}^{>0}$ and $\N^{<0}$.
\end{theorem}

\section{Explicit calculations of local Mellin transforms}\label{LMT state thms}

In this section we give precise statements of explicit formulas for calculating the local Mellin transforms and their inverses.  The results and proofs found in this chapter are analogous to those given for the local formal Fourier transforms in \cite{gs}.  Section \ref{LMT proofs} is devoted to proving the formulas given in section \ref{LMT state thms}.

\subsection{Calculation of $\mathcal{M}^{(0,\infty)}$}

\begin{theorem}\label{thmMzi} Let $s$ and $r$ be positive integers, $a\in\Bbbk-\{0\}$, and $f\in R^{\circ}_r(z)$ with $f=az^{-s/r}+\underline{o}(z^{-s/r})$.  Then
\[\mathcal{M}^{(0,\infty)}(E_f)\simeq D_g,\]
where $g\in {S^{\circ}_{s}(\t)}$ is determined by the following system of equations:
\begin{equation}\label{Mzisyseq1}f=-\t^{-1}
\end{equation}
\begin{equation}\label{Mzisyseq2} g=z-(-a)^{r/s}\left(\frac{r+s}{2s}\right)\t^{1+(r/s)}
\end{equation}
\end{theorem}

\begin{remark}
We determine $g$ using \eqref{Mzisyseq1} and \eqref{Mzisyseq2} as follows.
One can think of \eqref{Mzisyseq1} as an implicit definition for the variable $z$.  Thus we first use \eqref{Mzisyseq1} to give an explicit expression for $z$ in terms of $\t^{1/s}$.  We then substitute this explicit expression
into \eqref{Mzisyseq2} to get an expression for $g(\t)$ in terms of $\t^{1/s}$.  This same pattern for determining $g$ holds for similar calculations in this section.

When we use \eqref{Mzisyseq1} to write an expression for $z$ in terms of $\t^{1/s}$, the expression is not unique since we must make a choice of a root of unity.  More concretely, let $\eta$ be a primitive $s^{\text{th}}$ root of unity.  Then replacing $\t^{1/s}$ with $\eta\t^{1/s}$ in our explicit equation for $z$ will yield another possible expression for $z$.  This choice will not affect the overall result, however, since all such possible expressions will lie in the same Galois orbit.  Thus by Proposition \ref{diffopprop},\eqref{diffopprop1}, any choice of root of unity will correspond to the same difference operator.
\end{remark}

\begin{corollary*}\label{Mcorollary}

Let $E$ be an object in $\mathcal{C}_0^{>0}$.  By Proposition \ref{prop}, \eqref{prop3}, let $E$ have decomposition
$\d{E\simeq \bigoplus_i \bigg(E_{f_i}\otimes J_{m_i}\bigg)}$ where all $E_{f_i}$ have positive slope.
Then
$$\mathcal{M}^{(0, \infty)}(E)\simeq \bigoplus_i \bigg(D_{g_i}\otimes T_{m_i}\bigg)$$
where $D_{g_i}=\M^{(0,\infty)}(E_{f_i})$ for all $i$.
\end{corollary*}

\begin{proof}[Sketch of Proof]
   The equivalence of categories given in Theorem \ref{Mzi equiv} implies that  
$$\mathcal{M}^{(0, \infty)}\left[\bigoplus_i \bigg(E_{f_i}\otimes J_{m_i}\bigg)\right]\simeq \bigoplus_i \mathcal{M}^{(0, \infty)}\bigg(E_{f_i}\otimes J_{m_i}\bigg).$$
The equivalence also implies that $\mathcal{M}^{(0, \infty)}$ will map the indecomposable object $E_f\otimes J_m$ (as the unique indecomposable in $\mathcal{C}_0$ formed by $m$ successive extensions of $E_f$) to an indecomposable object $D_g\otimes T_m$ (as the unique indecomposable in $\mathcal{N}$ formed by $m$ successive extensions of $D_g$).  It follows that we only need to know how $\mathcal{M}^{(0, \infty)}$ acts on $E_f$, which is given by Theorem \ref{thmMzi}.
\end{proof}
\begin{remark}
Analogous corollaries hold for the calculation of the other local Mellin transforms, however we do not state them explicitly.
\end{remark}

\subsection{Calculation of $\mathcal{M}^{(x,\infty)}$}

\begin{theorem}\label{thmMxi} Let $s$ be a nonnegative integer, $r$ a positive integer, and $a\in\Bbbk - \{0\}$. Let $f\in R^{\circ}_r(z_x)$ with $f=az_x^{-s/r}+\underline{o}(z_x^{-s/r})$.  Then
\[\mathcal{M}^{(x,\infty)}(E_f)\simeq D_g,\]
where $g\in {S^{\circ}_{r+s}(\t)}$ is determined by the following system of equations:
\begin{equation*}\label{xisyseq1}f=-\left(\frac{z_x}{z}\right)\t^{-1}
\end{equation*}
\begin{equation*}\label{xisyseq2} g=z+\left(\frac{xs}{2(s+r)}\right)\t
\end{equation*}
\end{theorem}

\subsection{Calculation of $\mathcal{M}^{(\infty,\infty)}$}

\begin{theorem}\label{thmMii} Let $s$ and $r$ be positive integers and $a\in\Bbbk - \{0\}$. Then for $f\in R^{\circ}_r(\zeta)$ with $f=a\zeta^{-s/r}+\underline{o}(\zeta^{-s/r})$ we have
\[\mathcal{M}^{(\infty,\infty)}(E_f)\simeq D_g,\]
where $g\in {S^{\circ}_{s}(\t)}$ is determined by the following system of equations:
\begin{equation*}\label{Miisyseq1}f=-\t^{-1}
\end{equation*}
\begin{equation*}\label{Miisyseq2} g=z-(-a)^{r/s}\left(\frac{r+s}{2s}\right)\t^{1-(r/s)}
\end{equation*}
\end{theorem}

In section \ref{equiv cat sec} we explained that $\mathcal{M}^{-(0,\infty)}$, $\mathcal{M}^{-(x,\infty)}$, and $\mathcal{M}^{-(\infty,\infty)}$ are inverse functors for (respectively) $\mathcal{M}^{(0,\infty)}$, $\mathcal{M}^{(x,\infty)}$, and $\mathcal{M}^{(\infty,\infty)}$.  It follows that explicit formulas for the local inverse Mellin transforms can be found merely by ``inverting" the expressions found in Theorems \ref{thmMzi}, \ref{thmMxi}, and \ref{thmMii}.  We give an example below of what this would look like for $\mathcal{M}^{-(0,\infty)}$, the other local inverse Mellin transforms are similar.  The proofs are omitted.

\begin{theorem}\label{thmIMzi}
	 Let $p$ and $q$ be positive integers and $g\in S^{\circ}_q(\t)$ with $g=a\t^{p/q}+\underline{o}(\t^{p/q})$, $a\neq 0$.  Then
\[\mathcal{M}^{-(0,\infty)}(D_g)\simeq E_f,\]
where $f\in {R^{\circ}_{p}(z)}$ is determined by the following system of equations:
\begin{equation}\label{inv zisyseq1} g+a\left(\frac{p+q}{2q}\right)\t^{1+(p/q)}=z
\end{equation}
\begin{equation}\label{inv zisyseq2}f=-\t^{-1}
\end{equation}
\end{theorem}

\begin{remark}
We determine $f$ using \eqref{inv zisyseq1} and \eqref{inv zisyseq2} as follows.
First, using \eqref{inv zisyseq1} we explicitly express $\t$ in terms of $z^{1/p}$.  We then substitute this explicit expression for $\t$ 
into \eqref{inv zisyseq2} and solve to get an expression for $f(z)$ in terms of $z^{1/p}$.\\
\end{remark}

\section{Proof of theorems}\label{LMT proofs}

\subsection*{Outline}\label{Mellin outline}

We begin with a brief outline of the proof for Theorem \ref{thmMzi}.  Starting with Definition \ref{def inv mellin zi} of $\mathcal{M}^{(0,\infty)}$, we set $\t=-(z\nabla)^{-1}$ and $\Phi=z$.  For irreducible objects $E_f$ and $D_g$ we have $\nabla=\frac{d}{dz}+z^{-1}f$ and $\Phi=g\varphi$, and our goal is to use the given value of $f$ to find the expression for $g$.  Since $z=z(1)=\Phi (1)=g\varphi(1)=g$, this amounts to finding an expression for the operator $z$ in terms of the operator $\t$.  The equation $\t=-(z\nabla)^{-1}$ gives an expression for $\t$ in terms of $z$, and we use the Operator-root Lemma (\ref{abratlemma}) to write an explicit expression for the operator $z$ in terms of $\t$.  The calculation primarily involves finding particular fractional powers of $f$, but we must also keep track of the interplay between the linear and differential parts of $\nabla$ during the calculation; this interplay accounts for the subtraction of the term $(-a)^{r/s}\left(\frac{r+s}{2s}\right)\t^{1+\frac{r}{s}}$ from our expression for $g$.

 The proofs for Theorems \ref{thmMxi} and \ref{thmMii} are similar and thus outlines for their proofs are omitted.  The only change of note is that in the proof of Theorem \ref{thmMxi} we must also prove a separate case for when our connection is regular singular (i.e. when ord$(f)=0$). 

\begin{remark}
  We give a brief explanation regarding the origin of the system of equations found in Theorem \ref{thmMzi}.  Consider the equations in \eqref{eq mellin zi}.  Let $\nabla=z^{-1}f$ (i.e. as normally defined but without the differential part) and $\Phi=g$ (as normally defined but without the shift operator $\varphi$).  Then the equations $f=-\t^{-1}$ and $g=z$ fall out easily.  The reason the extra term shows up in \eqref{Mzisyseq2} is due to the interaction of the linear and differential parts of $\nabla$, as described above in the outline.
\end{remark}

\subsection{Proof of Theorem \ref{thmMzi}}

\begin{proof}
  Given $\t=-(z\nabla)^{-1}$ and $\nabla=\frac{d}{dz}+z^{-1}f$, we find that
\begin{equation}\label{thetaexp}
-\t=\left(z\frac{d}{dz}+f\right)^{-1}.
\end{equation}
We wish to express the operator $z$ in terms of the operator $\t$. 

Consider the equation 
\begin{equation}\label{thetaexp no diff}
-\t=f^{-1},
\end{equation}
 which is \eqref{thetaexp} without the differential part.  Equation \eqref{thetaexp no diff} can be thought of as an implicit expression for the variable $z$ in terms of the variable $\t$, which one can rewrite as an explicit expression $z=h(\t)\in \Bbbk((\t^{1/s}))$ for the variable $z$. Note that $h(\t)$ is not the same as the operator $z$.  Since the leading term of $f$ is $az^{-s/r}$, \eqref{thetaexp no diff} implies that $h(\t)=a^{r/s}(-\t)^{r/s}+\lo(\t^{r/s})$.  Similar reasoning and \eqref{thetaexp} indicate that the operator $z$ will be of the form
\begin{equation}\label{z with star}
  z=h(\t)+*(-\t)^{(r+s)/s}+\underline{o}(\t^{(r+s)/s}).
\end{equation}
Here the $*\in \Bbbk$ represents the coefficient that will arise from the interaction of the linear and differential parts of the operator $\t$.  We wish to find the value for *.  Let $A=f$ and $B=z\frac{d}{dz}$, then $[B,A]=zf'$.  From \eqref{thetaexp} we have $-\t=(A+B)^{-1}$, and we apply the Operator-root Lemma (\ref{abratlemma}) to find
\begin{equation}\label{theta r s}
    \begin{split}
  (-\t)^{\frac{r}{s}}& = f^{\frac{-r}{s}}-\left(\frac{r}{s}\right)f^{\frac{-r}{s}-1}z\left(\frac{\bigz}{zr}\right)-\frac{1}{2} \left(\frac{r}{s}\right) \left(-\frac{r}{s}-1\right)f^{\frac{-r}{s}-2}zf'+\lo(z^{(r+s)/r})\\
    & = (a^{-r/s}z+\dots)+a^{-(r+s)/s}\left(\frac{-\bigz}{s}+\frac{-(r+s)}{2s}\right)z^{(r+s)/r}+\lo(z^{(r+s)/r})\\
    & = a^{-r/s}\Big(z+\dots+a^{-1}\left[\frac{-\bigz}{s}+\frac{-(r+s)}{2s}\right]z^{1+(s/r)}+\lo(z^{1+(s/r)})\Big)
    \end{split}
\end{equation}
and
\begin{equation}\label{theta r plus s}
    (-\t)^{(r+s)/s}= a^{-1-(r/s)}z^{1+(s/r)}+\lo(z^{1+(s/r)}).
\end{equation}

\begin{remark}
    We use the notation $\frac{\bigz}{zr}$ to represent $\frac{d}{dz}$ since the operator $\frac{d}{dz}:K_r\to K_r$ acts on $z^{n/r}$ as $\frac{n}{rz}$ for all $n\in\bigz$.
\end{remark}

We can now find the value for * as follows.  Substituting the expressions from \eqref{theta r s} and \eqref{theta r plus s} into \eqref{z with star} and making a short calculation gives 
\[*=a^{r/s}\left[\frac{\bigz}{s}+\frac{r+s}{2s}\right]\]
and thus
\begin{equation}\label{z no star zi}
    z=h(\t)+a^{r/s}\left[\frac{\bigz}{s}+\frac{r+s}{2s}\right](-\t)^{(r+s)/s}+\underline{o}(\t^{(r+s)/s}).
\end{equation}
According to \eqref{z no star zi}, let us express $\hat{g}(\t)$ as 
\begin{equation}\label{g before isom}
\hat{g}(\t)=h(\t)-(-a)^{r/s}\left[\frac{\bigz}{s}+\frac{r+s}{2s}\right]\t^{(r+s)/s}+\underline{o}(\t^{(r+s)/s}).
\end{equation}
Since $h(\theta)=z$, by Proposition \ref{diffopprop}, \eqref{diffopprop1}, $M_{\hat{g}}$ will be isomorphic to $M_{g}$ where ${g}$ is as given in Theorem \ref{thmMzi}.
\end{proof}

\subsection{Proof of Theorem \ref{thmMxi}}

\begin{proof}

  Given $\t=-(z\nabla)^{-1}$ and $\nabla=\frac{d}{dz_x}+z_x^{-1}f$, we write $z=z_x+x$ and find that
\begin{equation}\label{three terms}
\begin{split}
-\t & = \left[(x+z_x)\left(\frac{d}{dz_x}+z_x^{-1}f\right)\right]^{-1}\\
	& = \left(zz_x^{-1}f+x\frac{d}{dz_x}+z_x\frac{d}{dz_x}\right)^{-1}
\end{split}
\end{equation}
Thus in the expression for $-\t^{-1}$ there are three terms. We handle the proof in two cases:

\noindent \textbf{Case One}: Regular singularity.  

In this case we have $f=\alpha\in \Bbbk-\{0\}$, $s=0$ and $r=1$.  Because $\alpha$ is only defined up to a shift by $\mathbb{Z}$ we can ignore the $\frac{d}{dz_x}$ term.  The remaining portion of the proof is as described in the remark following the outline in subsection \ref{Mellin outline}. Note that since $s=0$, the extra $\t$ term in \eqref{xisyseq2} will vanish.

\noindent \textbf{Case Two}: Irregular singularity.

In this situation we have ord$(f)<0$. As we shall see in the proof, the only terms in \eqref{three terms} that affect the final result are those of order less than or equal to -1 (with respect to $z_x$).  Specifically, since $z_x\frac{d}{dz_x}$ has order zero, all terms derived from it in the course of the calculations will fall into the $\lo(\t)$ term. Thus we can safely ignore the term $z_x\frac{d}{dz_x}$ for the remainder of the proof and consider only 

\begin{equation}\label{thetaexp xi}
  -\t = \left(zz_x^{-1}f+x\frac{d}{dz_x}\right)^{-1}.
\end{equation}

We wish to express the operator $z$ in terms of the operator $\t$.  The remainder of the proof is similar to the proof of Theorem \ref{thmMzi}, but we first solve for $z_x=z-x$ in terms of $\t$, then add $x$ to both sides to get an equation for $z$ alone.  

\end{proof}

\subsection{Proof of Theorem \ref{thmMii}}

\begin{proof}
Recall that $z=\frac{1}{\zeta}$ and $f\in\Bbbk((\zeta))$. Given $\t=-(z\nabla)^{-1}$ and $\nabla=-\zeta^2\frac{d}{d\zeta}+\zeta f$, we find that
\begin{equation}\label{thetaexpii}
-\t=\left(-\zeta\frac{d}{d\zeta}+f\right)^{-1}.
\end{equation}
We wish to express the operator $z$ in terms of the operator $\t$.  The proof is similar to that of Theorem \ref{thmMzi}, but first we find an expression for $\zeta$ in terms of $\t$, and then we will invert it.  

\end{proof}

\bibliographystyle{amsplain}
\bibliography{dissreferences}

\end{document}